\documentclass{amsart}
\usepackage[latin9]{inputenc}
\usepackage{amsthm}
\usepackage{amstext}
\usepackage{amssymb}
\usepackage{esint}

\makeatletter
\numberwithin{equation}{section}
\numberwithin{figure}{section}
\theoremstyle{plain}
\newtheorem{thm}{\protect\theoremname}
  \theoremstyle{definition}
  \newtheorem{defn}[thm]{\protect\definitionname}
  \theoremstyle{plain}
  \newtheorem{lem}[thm]{\protect\lemmaname}
  \theoremstyle{definition}
  \newtheorem{example}[thm]{\protect\examplename}
  \theoremstyle{remark}
  \newtheorem{rem}[thm]{\protect\remarkname}

\makeatother

  \providecommand{\definitionname}{Definition}
  \providecommand{\examplename}{Example}
  \providecommand{\lemmaname}{Lemma}
  \providecommand{\remarkname}{Remark}
\providecommand{\theoremname}{Theorem}

\begin{document}

\title{The convolution theorem of H\'ajek and Le Cam - revisited }

\author{Arnold Janssen and Vladimir Ostrovski}

\date{10.09.2013}
\begin{abstract}
The present paper establishes convolution theorems for regular estimators
when the limit experiment is non-Gaussian or of infinite dimension
with sparse parameter space. Applications are given for Gaussian shift
experiments of infinite dimension, the Brownian motion signal plus
noise model, Levy processes which are observed at discrete times and
estimators of the endpoints of densities with jumps. The method of
proof is also of interest for the classical convolution theorem of
H\'ajek and Le Cam. As technical tool we present an elementary approach
for the comparison of limit experiments on standard Borel spaces. 
\end{abstract}


\subjclass[2000]{Primary G2F10; secondary G2C05.}

\maketitle

\section{Introduction}

\label{intro} The classical convolution theorem of H\'ajek-Le Cam and
Inagaki, see H\'ajek(1970) and also Inagaki (1970), states that the
asymptotic distribution $Q$ of a sequences of properly normalized
regular estimators is a convolution product $Q=\nu\ast P$ of the
limit distribution $P$ of the asymptotically efficient estimators.
The convolution theorem was first proved under finite dimensional
local asymptotic normality (LAN). It is nowadays part of text books
like Bickel et al (1993), Pfanzagl (1994) and Witting and Müller-Funk
(1995), Le Cam and Yang (2000), for instance. The history was summerarized
by Le Cam, see Yang (1999). Le Cam pointed out that the case of infinite
dimensions beyond the Gaussian case is still open in general, see
also Le Cam (1994) and the recent discussion of Pözelberger, Schachermayer
and Strasser (2000). It is the aim of the present paper to reconsider
the convolution theorem again under very general conditions which
cover some results for infinite dimensions. Also we deal with the
estimation of infinite dimensional parameters when the parameter space
is not full. Based on Basu's theorem (about ancillary statistics)
we present another proof of the convolution which is also of interest
for classical situations. We will make some comments about different
kind of proofs of the convolution theorem. Bickel et al (1993) gave
an analytic proof via characteristic functions, see also Droste and
Wefelmeyer (1984). Another proof uses invariant means, see Pfanzagl
(1994) or van der Vaart (1988, 1989). In a certain sense this proof
is similar to Boll's proof of the convolution theorem for limit experiments,
see Boll (1955) or Strasser (1985), Sect. 38. Other proofs use the
Markov Kakutani fix point theorem, see Heyer (1979), Le Cam (1994)
or Pözelberger, Schachermayer and Strasser (2000). A very elegant
proof of the convolution theorem runs via limit experiments, see Le
Cam (1994), Theorem 1, Torgersen (1991), p. 401 ff and van der Vaart
(1991). Under LAN the sequence of underlying experiments converges
weakly to a limit shift experiment $\{P\ast\varepsilon_{\vartheta}:\vartheta\in\Theta\}$.
The (normalized) sequence of regular estimators converges to $Q\ast\varepsilon_{\vartheta}$
whenever $\vartheta$ is the true local parameter. Results about the
comparison of experiments show that $\{P\ast\varepsilon_{\vartheta}:\vartheta\in\Theta\}$
is more informative than $\{Q\ast\varepsilon_{\vartheta}:\vartheta\in\Theta\}$,
see also Section \ref{sec:4}. Then a Boll type convolution theorem
yields the result $Q=\nu\ast P$. Our approach is of this type but
only uses mostly elementary arguments.\\
 We will make some further comments about the literature. Early convolution
theorems for infinite dimensional parameter spaces were considered
by Moussatat (1976), Millar (1983, 1985), Strasser (1985) and Le Cam
(1994) mostly for cylinder measures. Van der Vaart (1991) gave a proof
for Gaussian shifts. Von den Heuvel and Klaassen (1999) and Sen (2000)
presented the connection to the Bayesian framework. Schick and Susarla
(1990) established an infinite dimensional convolution theorem with
application to the Kaplan Meier estimator. Further convolution results
about semiparametric estimation can be found in McNeney and Wellner
(2000). The connection between risk inequalities, spread inequalities
and the convolution theorem was discussed in Pfanzagl (2000). Beran
(1997) studied the relationship between bootstrap convergence and
the convolution theorem. On page 17 he mentioned that the convolution
theorem works analogously to Basu's theorem which is also central
in our proof.\\
 The paper is organized as follows. Section \ref{intro} introduces
first the main ideas for LAN sequences of experiments and univariate
parameters. This part is also of interest for teachers of statistics
courses since only elementary tools are required. As mentioned about
the proof is based on two steps. In a first step it is pointed out
that the asymptotic distributions $Q\ast\varepsilon_{\vartheta}$
of a regular sequence of estimators $T_{n}:\Omega_{n}\rightarrow\mathbb{R}$
lead to an experiment $\left\lbrace Q\ast\varepsilon_{\vartheta}:\vartheta\in\Theta\right\rbrace $
which is a randomization of a limit shift experiment (and less informative).
In the second step a comparison of $\left\lbrace Q\ast\varepsilon_{\vartheta}:\vartheta\in\Theta\right\rbrace $
with the limit experiment yields the convolution theorem. The main
ingredients of the proof are Le Cam's third Lemma and Basu's theorem
about ancillary statistics. As we will see in our sections \ref{sec:2}
and \ref{sec:4} this methodology and the present method of proof
also works for the asymptotically equivariant estimation of parameters
of infinite dimension. In this case the sample space and the parameter
space typically do not coincide and we have no full shift experiments.
Thus we take care about the convolution theorem when only a sparse
parameter space is available. Two applications are given for estimators
of infinite dimensional parameters. Example \ref{signal_noise_example}
deals with the convolution theorem for the signal plus noise model
with Brownian noise on $C\left[0,1\right]$. Example \ref{levy_pre_version}
works for non Gaussian Levy processes.\\
 Section \ref{preliminary_convolution} discusses variance inequalities
for unbiased estimators which may lead to asymptotic convolution theorems
whenever the asymptotic distributions are Gaussian. The core of the
paper is the comparison of limit experiments in Section \ref{sec:2}.
Under mild regularity assumptions a convolution theorem is established
for limit experiments. The proof of the convolution theorem is complete
if we are now citing Le Cam's theorem that limit experiments are more
informative than the limit $\left\lbrace Q\ast\varepsilon_{\vartheta}:\vartheta\in\Theta\right\rbrace $
of the estimators. This part is the topic of Section \ref{sec:4}.
Here we give an elementary proof of Le Cam's theorem for limit experiments
on standard Borel spaces. Moreover, the kernels required for the comparison
of experiments are constructed which transform one experiment in the
other one. Notice that these kernels are needed in Section \ref{sec:2}.
An interesting example is given when the endpoints of a distribution
with jumps of the densities have to be estimated.

Let us start with our introductory example. Let $(P_{t})_{t}$ be
a family of distributions on $\left(\Omega,\mathcal{A}\right)$ given
by a real parameter $t$. For an increasing sample size of $n$ independent
observations we like to estimate the parameter $t$ of the model.
Under regularity assumptions different estimators will be compared
locally around a fixed point $t=\vartheta_{0}$. Introduce by $t=\vartheta_{0}+\frac{\vartheta}{\sqrt{n}}$
a further local parameter $\vartheta\in\mathbb{R}$ and consider the
local model

\begin{equation}
P_{n,\vartheta}:=P_{\vartheta_{0}+\frac{\vartheta}{\sqrt{n}}}^{n}\label{lokal_model}
\end{equation}
on $\left(\Omega^{n},\mathcal{A}^{n}\right)$ with independent replications.\\

\textbf{Regularity of the model.} Let $\Theta\subset\mathbb{R}$ with
$0\in\Theta$ be the local parameter space. Suppose that the model
is local asymptotically normal (LAN), i.e. there exists some $\sigma>0$
and a central sequence of random variables $X_{n}:\Omega^{n}\rightarrow\mathbb{R}$
with

\begin{equation}
log\frac{dP_{n,\vartheta}}{dP_{n,0}}-[\vartheta X_{n}-\vartheta^{2}\sigma^{2}/_{2}]\rightarrow0\label{log_lacy_convergency}
\end{equation}
in $P_{n,0}$-probability for $\vartheta\in\Theta$, where

\begin{equation}
\mathcal{L}(X_{n}|P_{n,0})\rightarrow N(0,\sigma^{2})\text{ weakly }\label{where}
\end{equation}
holds as $n\rightarrow\infty$. Recall that LAN implies

\begin{equation}
\mathcal{L}(X_{n}|P_{n,\vartheta})\rightarrow(N(\vartheta\sigma^{2},\sigma^{2})\label{lan_implies}
\end{equation}
weakly under the parameter $\vartheta$ and the experiments

\begin{equation}
(\Omega^{n},\mathcal{A}^{n},\{P_{n,\vartheta}\})\rightarrow(\mathbb{R},\mathcal{B},\{N(\vartheta\sigma^{2},\sigma^{2})\})\label{experiments}
\end{equation}
are weakly convergent in the sense of Le Cam, see also Strasser (1985).\\

Consider now a sequence of estimators $T_{n}:\Omega^{n}\rightarrow\mathbb{R}$
of the parameter $t$. It is well known that the asymptotic efficiency
of $T_{n}$ typically implies that the linearization

\begin{equation}
\sqrt{n}(T_{n}-\vartheta_{0})-\frac{X_{n}}{\sigma^{2}}\rightarrow0\label{linearization}
\end{equation}
holds in $P_{n,0}$-probability. Recall that local asymptotic minimax
estimators and Fisher efficient estimators at $\vartheta_{0}$ have
this property. Recall that $T_{n}$ is called Fisher efficient if
$nVar_{\vartheta_{0}}(T_{n})$ reaches the asymptotic Cramer-Rao bound
of the experiment $\{P_{n,\vartheta}\}$.
\begin{defn}
A sequence of estimators $T_{n}$ is called $\Theta$-regular if

\begin{equation}
\mathcal{L}(\sqrt{n}(T_{n}-(\vartheta_{0}+\frac{\vartheta}{\sqrt{n}}))|P_{n,\vartheta})\rightarrow Q\label{theta_regular}
\end{equation}
converges weakly as $n\rightarrow\infty$ for all $\vartheta\in\Theta$
where $Q$ does not depend on $\vartheta$. 
\end{defn}
The classical H\'ajek-Le Cam convolution theorem states that the asymptotic
distribution $Q$ of a sequence of $\mathbb{R}$-regular estimators
is more spread out than $N(0,\frac{1}{\sigma^{2}})$ which is the
asymptotic distribution of $\frac{X_{n}}{\sigma^{2}}$. Although various
different proofs exist, also in text books, we will present a slightly
different proof which indicates the crucial steps of more general
convolution theorems with restricted parameter sets $\Theta$ of infinite
dimensions and also for non Gaussian limit experiments. Our method
of proof will first be presented for the classical one dimensional
LAN case.
\begin{thm}
\label{preliminary_convolution} Suppose that LAN holds for the parameter
set $\Theta\subset\mathbb{R}$ where the closure of $\Theta$ has
non empty interior, $\stackrel{\circ}{\overline{\Theta}}\not=\emptyset$.
Let $T_{n}$ be a sequence of $\Theta$-regular estimators, i.e. $\mathcal{L}(\sqrt{n}(T_{n}-\vartheta_{0})|P_{n,\vartheta}))\rightarrow Q\ast\varepsilon_{\vartheta}$
holds for each $\vartheta\in\Theta$. Then

(a) For each $\vartheta$ the vector of random variables

\begin{equation}
(X_{n},\sqrt{n}(T_{n}-\vartheta_{0})-\frac{X_{n}}{\sigma^{2}})\label{random_variables}
\end{equation}
are convergent in distribution to a pair of independent random variables
$(X,Z)$ on $\mathbb{R}^{2}$. Let $\nu$ denote the weak limit distribution
of the second component $Z$ under $\vartheta=0$, i.e.

\begin{equation}
\mathcal{L}(\sqrt{n}(T_{n}-\vartheta_{0})-\frac{X_{n}}{\sigma^{2}}|P_{n,0})\rightarrow\nu.\label{second_component}
\end{equation}
Then (\ref{second_component}) also holds under $P_{n,\vartheta}$
with the same limit $\nu$ independent of $\vartheta\in\Theta$.

(b) For fixed $\vartheta$ the limit distribution of (\ref{random_variables})
is given by

\begin{equation}
N(\vartheta\sigma^{2},\sigma^{2})\otimes\nu\label{lim_distribution_given_by}
\end{equation}

and the convolution theorem

\begin{equation}
Q\ast\varepsilon_{\vartheta}=\nu\ast N(\vartheta,\frac{1}{\sigma^{2}})\label{convolution_theorem}
\end{equation}
holds for all $\vartheta\in\Theta$.

(c) The sequence $T_{n}$ is asymptotically efficient at $\vartheta_{0}$,
i.e. (\ref{linearization}) holds in $P_{n,0}$ probability, iff the
weak limit $Q\ast\varepsilon_{\vartheta}$ of $\sqrt{n}(T_{n}-\vartheta_{0})\;$
w.r.t. $\; P_{n,\vartheta}$ is $N(\vartheta\sigma^{2},\sigma^{2})$
at least for one $\vartheta\in\Theta$. 
\end{thm}
An elegant proof of the convolution theorem is based on the third
Lemma of Le Cam which is summarized for LAN experiments.
\begin{lem}
\label{conseq_third_lecam_lemma} Suppose that $S_{n}:\Omega^{n}\rightarrow\mathbb{R}$
is a further sequence of statistics with weak limit law

\begin{equation}
\mathcal{L}((X_{n},S_{n})|P_{n,o})\rightarrow\mu_{0}\label{weak_limit_law}
\end{equation}
on $\mathbb{R}^{2}.$ Then (\ref{weak_limit_law}) is weakly convergent
also under each sequence $P_{n,\vartheta}$ to the distribution $\mu_{\vartheta}$
on $\mathbb{R}^{2}$ with $\mu_{\vartheta}\ll\mu_{0}$ and

\begin{equation}
\frac{d\mu_{\vartheta}}{d\mu_{0}}(x,y)=exp(\vartheta x-\frac{1}{2}\vartheta^{2}\sigma^{2}).\label{preparations_end}
\end{equation}

\end{lem}
Observe that formula (\ref{preparations_end}) is a direct consequence
of Le Cam's third lemma of the form given by H\'ajek, $\check{S}$idak
and Sen (1999), 7.1.4 and the references given in the proof of Lemma
\ref{lem:4:4} below. After these preparations we will indicate the
crucial steps of the proof of the convolution theorem.
\begin{proof}
\emph{(of Theorem }\ref{preliminary_convolution}) 

Step 1. According to (\ref{where}) and (\ref{theta_regular}) the
pair of random variables (\ref{second_component}) is tight on $\mathbb{R}^{2}$
for $\vartheta=0$. Then we find a subsequence $\{m\}\subset\mathbb{N}$
such that

\begin{equation}
(X_{m},\sqrt{m}(T_{m}-\vartheta_{0}))\rightarrow(X,T)\label{subsequence}
\end{equation}
is distributional convergent to some $(X,T)$ under $P_{n,0}$ along
our subsequence. If $\mu_{0}$ denotes the distribution of $(X,T)$
under $\vartheta=0$ then Le Cam's third Lemma implies that the distributional
convergence along $\{m\}$ of (\ref{subsequence}) also holds under
$P_{n,\vartheta}$ with limit law $\mu_{\vartheta}$ given by (\ref{preparations_end})
for each $\vartheta\in\Theta$. Let $\Pi_{i}:\mathbb{R}^{2}\rightarrow\mathbb{R}$
denote the projections on the coordinates for $i=1,2$. By (\ref{where})
and (\ref{theta_regular}) we have

\begin{equation}
\mathcal{L}(\frac{\Pi_{1}}{\sigma^{2}}|\mu_{\vartheta})=N(\vartheta,\frac{1}{\sigma^{2}})\text{ and }\mathcal{L}(\Pi_{2}\vert\mu_{\vartheta})=Q\ast\varepsilon_{\vartheta}\label{projections_convolution}
\end{equation}
for all $\vartheta\in\Theta$.

Step 2. By Neyman's criterion about sufficiency the projection $\Pi_{1}$
is $\{\mu_{\vartheta}:\vartheta\in\Theta\}$ sufficient, confer (\ref{preparations_end}).
Within the language of comparison experiments this means that

\[
\{N(0,\frac{1}{\sigma_{2}})\ast\varepsilon_{\vartheta}:\vartheta\in\Theta\}
\]
is more informative as $\{Q\ast\varepsilon_{\vartheta}:\vartheta\in\Theta\}.$
For full parameter spaces $\Theta=\mathbb{R}$ the convolution theorem
(\ref{convolution_theorem}) follows from Boll's convolution theorem
of shift experiments. Boll's theorem requires an analytic proof which
sometimes uses fix point methods. We will substitute this part by
more elementary arguments which give us the full result also when
$\Theta$ is not full. For the details we refer to the proof of Theorem
\ref{extended_convolution_theorem}. The proof runs as follows. It
is easy to see that

(i) $\Pi_{1}$ is sufficient and boundedly complete for $\{\mu_{\vartheta}:\vartheta\in\Theta\}$.
Observe that the densities of $\mu_{\vartheta}$ can be extended to
the parameter space $\stackrel{\circ}{\overline{\Theta}}$ and that
the bounded completeness has to be checked for $\stackrel{\circ}{\overline{\Theta}}$
only by continuity arguments, see Remark \ref{rem:3.4}. Without restrictions
we may assume that $0\in\stackrel{\circ}{\overline{\Theta}}$ holds.
Otherwise the family (\ref{preparations_end}) can be shifted due
to the translation invariance of the Gaussian shift family. Also we
can find an open set $\Theta_{1}\subset\stackrel{\circ}{\overline{\Theta}}$
with $\Theta_{1}+\Theta_{1}\subset\stackrel{\circ}{\overline{\Theta}}$
where $\left\lbrace \mu_{\vartheta}:\vartheta\in\Theta_{1}\right\rbrace $
is boundedly complete as required in the proof of Theorem \ref{extended_convolution_theorem}.
\\

(ii) It is not hard to check that $\Pi_{2}-\frac{\Pi_{1}}{\sigma^{2}}$
is ancillary w.r.t. $\Theta_{1}$. Details can be found in (\ref{eq:2.10})-(\ref{eq:2.12})
of the proof of Theorem \ref{extended_convolution_theorem} below.
Consequentely, Basu's theorem, see Pfanzagl (1994), implies that

\begin{equation}
\Pi_{1}\text{ and }\Pi_{2}-\frac{\Pi_{1}}{\sigma^{2}}\label{Basu_theorem_implies}
\end{equation}
are independent under $\mu_{\vartheta}$ for each $\vartheta\in\Theta_{1}$,
i.e. (\ref{lim_distribution_given_by}) and the convolution theorem
(\ref{convolution_theorem}) hold.\\
 By the consideration of Fourier transform we see that the factor
$\nu$ of (\ref{convolution_theorem}) is unique. If we now turn to
(\ref{Basu_theorem_implies}) we see that the limit distributions
(\ref{subsequence}) must be the same $\mu_{0}$ for all subsequences
$\{m\}$ when (\ref{subsequence}) holds. By tightness this fact proves
the convergence of (\ref{subsequence}) along $n\in\mathbb{N}$ first
for $\vartheta=0$ and by Lemma \ref{conseq_third_lecam_lemma} for
all $\vartheta$. These arguments finish the proof of part (a) and
(b). Part (c) is now trivial since the efficiency of $T_{n}$ corresponds
to $\nu=\varepsilon_{0}$. 
\end{proof}
Along these lines of the proof more general convolution theorems will
be established.\\
 - The sample space and the parameter space may be of infinite dimension.\\
 - The parameter set $\Theta$ may be a restricted set (not a full
vector space).\\
 - If we turn to limit experiments other distributions than Gaussian
distributions are allowed.\\
The extended convolution theorems of Section \ref{sec:2} have various
application. As application we will consider two examples with parameter
and sample spaces of infinite dimension which can be treated by the
new convolution theorem.
\begin{example}
{[}Signal plus noise model with fixed finite sample size{]}\label{signal_noise_example}
Let $B_{1},\ldots,B_{n}$ denote $n$ independent standard Brownian
motions with values in $C[0,1]$. Our observations are the processes

\begin{equation}
X_{i}(t)=B_{i}(t)+\int_{0}^{t}\vartheta(u)du,\;\;0\leq t\leq1,\:1\leq i\leq n,\label{signal_noise}
\end{equation}
where the parameter $\vartheta$ belongs to a class square integrable
functions $\Theta\subset L_{2}[0,1]$ with respect to the uniform
distribution on $[0,1]$. We like to estimate the signal $h(\vartheta)$,

\begin{equation}
h:\Theta\longrightarrow C[0,1],\: h(\vartheta)(t)=\int_{0}^{t}\vartheta(u)du\label{signal}
\end{equation}
by estimators $T(X_{1},\ldots,X_{n})$,

\begin{equation}
T:C[0,1]^{n}\longrightarrow C[0,1].\label{estimators}
\end{equation}
The natural estimator is $S$,

\begin{equation}
S(X_{1},\ldots,X_{n})=\frac{1}{n}\sum_{i=1}^{n}X_{i}.\label{natural_estimator}
\end{equation}
At a fixed point $\vartheta_{0}\in\Theta$ this estimator can be compared
with competing estimators $T$. For this purpose introduce in addition
to the global parameter $\vartheta_{0}$ another local parameter $\eta\in L_{2}[0,1]$
by

\begin{equation}
\vartheta=\vartheta_{0}+\frac{\eta}{\sqrt{n}}.\label{another_local_parameter}
\end{equation}
Two different classes of estimator will be studied.

(a) (Equivariant estimation). Suppose that the estimator fulfills

\begin{equation}
\mathcal{L}\left(\sqrt{n}(T-h(\vartheta_{0}))\left|\vartheta_{0}+\frac{\eta}{\sqrt{n}}\right.\right)=\mathcal{L}\left(\left.\sqrt{n}\left(T-h(\vartheta_{0})\right)\right|\vartheta_{0}\right)\ast\varepsilon_{\eta}\label{estimator_fulfils}
\end{equation}
for $\eta$ with $\vartheta_{0}+\frac{\eta}{\sqrt{n}}\in\Theta.$
Under mild assumptions about the size of $\Theta$ the convolution
theorem

\begin{equation}
\mathcal{L}\left(\left.{\sqrt{n}}\left(T-h(\vartheta_{0})\right)\right|\vartheta_{0}\right)=\nu\ast\mathcal{L}\left(\left.\sqrt{n}\left(S-h(\vartheta_{0})\right)\right|\vartheta_{0}\right)\label{signal_noise_convolution}
\end{equation}
holds on $(C[0,1])$, see Example \ref{examp:3:6} below for details.\\

(b) (Unbiased estimation). Under various assumption it can be shown
that $S$ is the best unbiased estimator which is Fisher efficient
in the sense that $S$ attains the nonparametric Cramer-Rao bound,
see Janssen (2003) for related results.
\end{example}
Recall that a real Levy process $\left(Z_{t}\right)_{t\geq0}$ is
a stochastically continuous process with independent stationary increments. 
\begin{example}
\label{levy_pre_version} Let $\left(Z_{t}\right)_{t\geq0}$ be a
Levy process with absolutely continuous distributions $\mathcal{L}\left(Z_{t}\right)\ll\lambda$
for all $t>0$. At a given sequence of discrete times $0<t_{1}<t_{2}<\ldots$
the Levy process serves as error distribution of our observations
\begin{equation}
X=\left(Z_{t_{i}}+\vartheta_{i}\right)_{i\in\mathbb{N}}\in\mathbb{R}^{\mathbb{N}}\label{levy_observations}
\end{equation}
with unknown parameters $\vartheta=\left(\vartheta_{i}\right)_{i\in\mathbb{N}}\in\Theta\subset\mathbb{R}^{\mathbb{N}}$.
We are going to estimate the parameter $\vartheta$ by equivariant
estimators $T(X),T:\mathbb{R}^{\mathbb{N}}\rightarrow\mathbb{R}^{\mathbb{N}}$,
with 
\begin{equation}
\mathcal{L}\left(T\left(X\right)|\vartheta\right)=\mathcal{L}(T(X)|0)\ast\varepsilon_{\vartheta}\label{invariant_estimators}
\end{equation}
for $\vartheta\in\Theta$. Under the assumptions the identity $S=\mbox{id}$
is the best equivariant estimator in the sense that the convolution
theorem 
\begin{equation}
\mathcal{L}(T(X)|\vartheta)=\nu\ast\mathcal{L}(S(X)|\vartheta)\label{ct:3}
\end{equation}
holds for all $\vartheta\in\Theta$. The details are presented in
Example \ref{example_levy_pro}.\end{example}
\begin{defn}
Suppose that there exist some $\vartheta^{0}=\left(\vartheta_{i}^{0}\right)_{i\in\mathbb{N}}\in\mathbb{Q}^{\mathbb{N}}$,
such that 
\begin{equation}
\left\lbrace \left(\vartheta_{i}\right)_{i\in\mathbb{N}}\in\mathbb{Q}^{\mathbb{N}}:\vartheta_{i}=\vartheta_{i}^{0}\:\mbox{finally}\right\rbrace \subset\Theta\label{ct:4}
\end{equation}
holds. 
\end{defn}

\section{Unbiased estimation, variance inequalities and preliminary versions
of the convolution theorem}

In this section we will start with finite sample results for locally
unbiased estimation of vector valued statistical functionals. It is
shown that variance inequalities within convex classes of estimators
are linked to preliminary versions of the convolution theorem. Moreover,
it is shown that the existence of a sequence of locally minimum variance
estimators already imply a convolution theorem whenever the underlying
estimators are jointly asymptotically normal.

Let $g:\Theta\to W$ denote a statistical functional with values in
a real vector space $W$. Suppose that $I\subset\left\{ f:W\to\mathbb{R}\right\} $
is a set of linear functions and let $\mathcal{B}:=\sigma(f:f\in I)$
denote the $\sigma$-field generated by $I$ on $W$. Let $E=\left(\Omega,\mathcal{A},\left\lbrace P_{\vartheta}:\vartheta\in\Theta\right\rbrace \right)$
denote the underlying experiment. Throughout, we like to estimate
$g$ by $(\mathcal{A},\mathcal{B})$-measurable estimators 
\begin{equation}
T:(\Omega,\mathcal{A})\rightarrow(W,\mathcal{B}).\label{eq:1.1}
\end{equation}
Different estimators will now be compared at a fixed point $\vartheta_{0}\in\Theta$.
Locally at this point we will consider unbiased estimation of $g\left(\vartheta_{0}\right)$.
Assume that 

(A) $\mathcal{K}$ is a class of estimators $T:\;\Omega\to W$ with
$E_{\vartheta_{0}}(f(T))=f\left(g\left(\vartheta_{0}\right)\right)$
and $\mathrm{Var}_{\vartheta_{0}}(f(T))<\infty$ for all $f\in I$
such that $\mathcal{K}$ has the extended convexity property $a\mathcal{K}+(1-a)\mathcal{K}\subset\mathcal{K}$
for all $a\in\mathbb{R}$. \\
Observe that the distribution of an estimator $T$ is completely specified
by the distribution of the process $f(T)_{f\in I}$. 
\begin{lem}
\label{lem:1.1} Consider an estimator $S\in\mathcal{K}$ where assumption
(A) holds for $\mathcal{K}$. Then the following assertions are equivalent:

$\mbox{Var}_{\vartheta_{0}}\left(f\left(S\right)\right)=\min\limits _{T\in\mathcal{K}}\mbox{Var}_{\vartheta_{0}}\left(f\left(T\right)\right)$
~~~for all $f\in I$. 

For each $T\in\mathcal{K}$ the processes $\left(f\left(S\right)\right)_{f\in I}$
and $\left(f\left(T-S\right)\right)_{f\in I}$ are uncorrelated at
$\vartheta_{0}$,\\
 i.e. $\mbox{Cov}_{\vartheta_{0}}\left(f\left(S\right),f\left(T\right)-f\left(S\right)\right)=0$
for all $f\in I$. In this case we have 
\[
\mbox{Var}_{\vartheta_{0}}\left(f\left(T\right)\right)=\mbox{Var}_{\vartheta_{0}}\left(f\left(S\right)\right)+\mbox{Var}_{\vartheta_{0}}\left(f\left(T-S\right)\right).
\]
\end{lem}
\begin{proof}
We fix $f\in I$ and $\vartheta_{0}\in\Theta$. Suppose that (a) holds.
The estimator $T_{t}:=S+t\left(T-S\right)=\left(1-t\right)S+tT$ is
in $\mathcal{K}$ for all $t\in\mathbb{R}$. The function $t\mapsto\mbox{Var}_{\vartheta_{0}}\left(f\left(T_{t}\right)\right)$
has the minimum in $t=0$. This implies
\begin{eqnarray*}
0 & = & \left.\frac{\partial}{\partial t}\mbox{Var}_{\vartheta_{0}}\left(f\left(T_{t}\right)\right)\right|_{t=0}\\
 & = & \left.\frac{\partial}{\partial t}\left(\mbox{Var}_{\vartheta_{0}}\left(f\left(S\right)\right)+t^{2}\mbox{Var}_{\vartheta_{0}}\left(f\left(T-S\right)\right)+2t\mbox{Cov}_{\vartheta_{0}}\left(f\left(S\right),f\left(T\right)-f\left(S\right)\right)\right)\right|_{t=0}\\
 & = & 2\mbox{Cov}_{\vartheta_{0}}\left(f\left(S\right),f\left(T\right)-f\left(S\right)\right).
\end{eqnarray*}
Suppose that (b) holds. Then for each $S\in\mathcal{K}$ the equality
\[
0=\mbox{Cov}_{\vartheta_{0}}\left(f\left(S\right),f\left(T\right)-f\left(S\right)\right)=\mbox{Cov}_{\vartheta_{0}}\left(f\left(S\right),f\left(T\right)\right)-\mbox{Var}_{\vartheta_{0}}\left(f\left(S\right)\right).
\]
follows. As a consequence we have
\[
\mbox{Var}_{\vartheta_{0}}\left(f\left(S\right)\right)=\mbox{Cov}_{\vartheta_{0}}\left(f\left(S\right),f\left(T\right)\right)\leq\mbox{Var}_{\vartheta_{0}}\left(f\left(S\right)\right)^{\frac{1}{2}}\mbox{Var}_{\vartheta_{0}}\left(f\left(T\right)\right)^{\frac{1}{2}}
\]
and the inequality $\mbox{Var}_{\vartheta_{0}}\left(f\left(S\right)\right)\leq\mbox{Var}_{\vartheta_{0}}\left(f\left(T\right)\right)$
follows. 
\end{proof}
The present result is a slight but useful extension of Rao's covariance
method, see Lehmann (1983), p. 77. 
\begin{example}
The following classes $\mathcal{K}$ of estimators have the extended
convexity property: 

(a)The estimators $T$ where $f(T)$ is unbiased for our functional
$g$ and all $\vartheta\in\Theta$. 

(b)Let $\Theta$ be a subset of a linear space $V$ such that $g$
has an extension $g:V\rightarrow W$ as linear function. Suppose that
$\Omega=V$ and let $T$ be the strictly equivariant estimators (for
$g$), i.e. 
\[
T(x+\vartheta)=T(x)+g(\vartheta)
\]
for all $x\in V$ and all $\vartheta\in\Theta$. 
\end{example}
Lemma \ref{lem:1.1} is a preliminary version of a convolution theorem
which is expressed by the variance decomposition. The convolution
theorem 
\begin{equation}
\mathcal{L}\left(T|\vartheta_{0}\right)=\mathcal{L}\left(S|\vartheta_{0}\right)\ast\mathcal{L}\left(T-S|\vartheta_{0}\right)\label{eq:1.4}
\end{equation}
holds whenever $(f(S))_{f}$ and $(f(T-S))_{f}$ are independent under
$\vartheta_{0}$ for each $f\in I$. This will not be true in general
but if the vector $(f(S),f(T-S))$ is jointly Gaussian then the components
are independent since they are uncorrelated. If $I$ is a linear space
of functions the convolution theorem for the $f$-marginals then implies
(\ref{eq:1.4}). This simple observation leads to an asymptotic convolution
theorem of asymptotically normal locally unbiased estimators.

Consider a sequence of experiments 
\begin{equation}
E_{n}=\left(\Omega_{n},\mathcal{A}_{n},\left\{ P_{n,\vartheta}\;:\;\vartheta\in\Theta\right\} \right)\label{eq:1.5}
\end{equation}
and a sequence $g_{n}:\;\Theta\to W$ of statistical functionals.
For each $n$ let $\mathcal{K}_{n}$ be a class of unbiased estimators
of $g_{n}$ at $\vartheta_{0}$ so that assumption (A) holds. Suppose
as in Lemma \ref{lem:1.1} (a) that $S_{n}$ is a minimum variance
estimator at $\vartheta_{0}$ in $\mathcal{K}_{n}$ for each $n$.
\begin{thm}
\label{limit_convolution} Suppose that $I$ is a linear subspace
of functions on $W$. Let $T_{n}\in\mathcal{K}_{n}$ be a competing
sequence of estimators such that the joint distribution 
\begin{equation}
a_{n}\left(f\left(S_{n}\right)-E_{\vartheta_{0}}\left(f\left(S_{n}\right)\right),f\left(T_{n}\right)-E_{\vartheta_{0}}\left(f\left(T_{n}\right)\right)\right)\rightarrow\left(f(S),f(T)\right)\label{eq:1.6}
\end{equation}
is weakly convergent under $\vartheta_{0}$ to a centered Gaussian
random variable $\left(f(S),f(T)\right)$ for all $f\in I$, where
$a_{n}>0$ denotes a normalizing sequence. Assume that $S$ and $T$
are $W$-valued random variables. Let in addition 
\begin{equation}
\mathrm{Var}_{\vartheta_{0}}\left(a_{n}f\left(S_{n}\right)\right)\rightarrow\mathrm{Var}\left(f(S)\right)\mbox{ and }\mathrm{Var}_{\vartheta_{0}}\left(a_{n}f\left(T_{n}\right)\right)\rightarrow\mathrm{Var}\left(f(T)\right)\label{eq:1.7}
\end{equation}
hold for all $f\in I$ as $n\to\infty$. Then the convolution theorem
in the sense of equation (\ref{eq:1.4}) holds for the asymptotic
distributions on $W$. \end{thm}
\begin{proof}
In a first step the convolution theorem 
\begin{equation}
\mathcal{L}\left(f(T)|\vartheta_{0}\right)=\mathcal{L}\left(f(S)|\vartheta_{0}\right)\ast\mathcal{L}\left(f(T-S)|\vartheta_{0}\right)\label{t2:6}
\end{equation}
is proved for all univariate marginals given by $f\in I$. For this
purpose it is enough to show that 
\begin{equation}
\mathrm{Cov}\left(f(S),f(T-S)\right)=0\label{t2:7}
\end{equation}
since $(f(S),f(T)-f(S))$ is Gaussian. On a new probability space
we may find random variables $X_{n},X,Y_{n}$ and $Y$ with distributions
\[
\mathcal{L}\left(X_{n},Y_{n}\right)=\mathcal{L}\left(a_{n}\left(f(T_{n})-E(f(T_{n}))\right),a_{n}\left(f(S_{n})-E(f(S_{n}))\right)\right)
\]
and 
\[
\mathcal{L}(X,Y)=\mathcal{L}(f(T),f(S))
\]
with $(X_{n},Y_{n})\rightarrow(X,Y)$ almost surely, see Dudley (1989),
p.325. Our assumptions together with (\ref{eq:1.7}) imply by Vitali's
theorem the $L_{2}$-convergence of $X_{n}\rightarrow X$ and $Y_{n}\rightarrow Y$.
Thus 
\[
E((X_{n}-Y_{n})Y_{n})\rightarrow E((X-Y)Y)=\mathrm{Cov(f(S),f(T-S))}
\]
holds. Since $f(S_{n})$ is a minimum variance estimator the covariance
principle of Lemma \ref{lem:1.1} implies 
\[
E((X_{n}-Y_{n})Y_{n})=0
\]
for each $n$ and (\ref{t2:7}) holds. Since $I$ is a linear space
the Cramer Wold device then implies via (\ref{t2:6}) the convolution
equation (\ref{eq:1.4}). \end{proof}
\begin{rem}
Under the assumptions of Theorem \ref{limit_convolution} the univariate
marginals $f(S)$ and $f(T-S)$ are independent for all $f\in I$. \end{rem}
\begin{example}
Let $\mathcal{P}$ be a set of probability measures on a measurable
space $(\Omega,\mathcal{A})$ and let $X$ be a $W$-valued random
variable. Suppose that $g:\;\mathcal{P}\to W$ is the mean functional
in the sense that $f(g(P))=E_{P}(f(X))$ holds for all $f\in I$ and
every $P\in\mathcal{P}$. We like to estimate $g_{n}\left(P^{n}\right):=g(P)$
for $\mathcal{P}_{n}:=\left\{ P^{n}\;:\; P\in\mathcal{P}\right\} $.
The model is given by independent copies $X_{1}$, $X_{2}$, {\dots}
of $X$ with unknown law $\mathcal{L}\left(X_{1},\dots,X_{n}\right)=P^{n}$
for each $n$. Consider the class $\mathcal{K}_{n}$ of $W$-valued
unbiased estimators $T_{n}=T_{n}\left(X_{1},\dots,X_{n}\right)$ with
$E_{P}\left(f\left(T_{n}\right)\right)=f(g(P))$ for all $f\in I$
and all $P\in\mathcal{P}$. Let $P_{0}\in\mathcal{P}$ be fixed. Let
$I$ always denote a linear space of functions. 

(a) Let $\mathcal{P}$ be a convex set such that the marginal distributions
$\left\{ \mathcal{L}(f|P)\;:\; P\in\mathcal{P}\right\} $ are complete
for each $f\in I$. Then the $W$-valued mean $S_{n}=\frac{1}{n}\sum_{i=1}^{n}X_{i}$
is a minimum variance estimator in $\mathcal{K}_{n}$ at $P_{0}$
in the sense of Lemma \ref{lem:1.1}(a). Recall that the order statistics
of $f\left(X_{1}\right),\dots,f\left(X_{n}\right)$ are sufficient
and complete, see Pfanzagl (1994), Sect. 1.5. Thus the theorem of
Lehmann and Scheffe can be applied. 

(b) Let $T_{n}$ be a competing sequence of estimators in $\mathcal{K}_{n}$
which admit a linearization at $P_{0}^{n}$ with 
\begin{equation}
n\mathrm{Var}_{P_{0}^{n}}\left(f\left(T_{n}-\sum_{i=1}^{n}a_{ni}X_{i}\right)\right)\longrightarrow0
\end{equation}
as $n\to\infty$ for all $f$. Assume that $a_{ni}$ are reals with
\begin{equation}
n\sum_{i=1}^{n}a_{ni}^{2}\rightarrow\kappa>0\,\mbox{ and }\, n^{\frac{1}{2}}\max_{1\leq i\leq n}\left|a_{ni}\right|\rightarrow0\:\mbox{as \ensuremath{n\to\infty}.}
\end{equation}
Then together with the Cramer-Wold device the central limit theorem
of Lindeberg and Feller implies that 
\begin{equation}
\left(\sqrt{n}\left(f\left(S_{n}\right)-g\left(P_{0}\right)\right),\sqrt{n}\left(f\left(T_{n}\right)-g\left(P_{0}\right)\right)\right)\longrightarrow\left(Y_{f}^{(S)},Y_{f}^{(T)}\right)
\end{equation}
is weakly convergent under $P_{0}^{n}$ to a centered Gaussian random
variable for all $f\in I$. On the space $\left(\mathbb{R}^{I},\mathcal{B}^{I}\right)$
the convolution theorem 
\begin{equation}
\mathcal{L}\left(\left(Y_{f}^{(T)}\right)_{f\in I}\right)=\mathcal{L}\left(\left(Y_{f}^{(S)}\right)_{f\in I}\right)\ast\mathcal{L}\left(\left(Y_{f}^{(T)}-Y_{f}^{(S)}\right)_{f\in I}\right)
\end{equation}
holds. If in addition $Y_{f}^{(S)}=f(S)$, $Y_{f}^{(T)}=f(T)$ arise
from $W$-valued random variables $S$ and $T$ then the convolution
theorem (\ref{eq:1.4}) holds on $W$. 

(c) Let $\mathcal{P}$ denote all continuous distributions on the
unit interval $[0,1]$. We estimate the distribution function $g(P)=F$
of $P$ in the space of cad lag functions $W=D[0,1]$ on $[0,1]$.
The set $I$ is the linear space generated by the projections $h\mapsto h(t)$
for $t\in[0,1]$. The minimum variance estimator is the empirical
distribution function $S_{n}=\hat{F}_{n}$. The Gaussian limit process
$S$ on $W$ is the transformed Brownian bridge $t\mapsto B_{0}(F(t))$,
where $B_{0}(\cdot)$ denotes a standard Brownian bridge. It can be
shown that under our regularity assumptions the limit process of $t\mapsto\sqrt{n}\left(T_{n}(t)-F(t)\right)$
can be realized by some process $\left(T(t)\right)_{t}$ in $D[0,1]$
and under the conditions of part (b) the convolution theorem holds
for $T$ and $S$ on $D[0,1].$ 
\end{example}

\section{The convolution theorem for limit experiments}

\label{sec:2}

In general there will not be finite sample optimal estimators (or
they turn out to be unknown) which may serve as benchmark for the
underlying sequence $T_{n}$. At this stage an asymptotic solution
is presented within the limit experiment where a convolution theorem
can be presented under fairly general condition.

In a first step linear functionals are estimated where the parameter
space $\Theta\subset V_{0}$ is part of a vector space $V_{0}$. Below
let $V,W$ always denote measurable vector spaces with $\sigma$-fields
$\mathcal{B}(V)$ and $\mathcal{B}(W)$ where $\left(W,\mathcal{B}(W)\right)$
is a standard Borel space (i.e.~$W$ is a measurable set of a polish
space and $\mathcal{B}(W)$ is the Borel $\sigma$-field).

In contrast to the finite dimensional case the parameter space and
the sample spaces do not coincide, see Example \ref{signal_noise_example}
where $V_{0}=L_{2}[0,1]$ is a Hilbert space and $V=C[0,1]$ is the
path space of Brownian motion. However, $V_{0}$ can be embedded in
$V$. Assume that there exists a linear injective function 
\begin{equation}
h:\quad V_{0}\longrightarrow V\label{eq:2.1}
\end{equation}
and let $f:\; V\to W$ be a linear measurable function. Then we are
going to estimate the functional $g$ 
\begin{equation}
g:\quad\Theta\longrightarrow W\quad,\qquad g:=f\circ h_{|\Theta}\quad.\label{eq:2.2}
\end{equation}
Under regularity assumptions the limit experiment of $E_{n}$, see
(\ref{eq:1.5}), has the form 
\begin{equation}
E=\left(V,\mathcal{B}(V),\left\{ P\ast\varepsilon_{h(\vartheta)}\;:\;\vartheta\in\Theta\right\} \right)\label{eq:2.3}
\end{equation}
for some distribution $P$ on $\left(V,\mathcal{B}(V)\right)$, see
Section \ref{sec:4} for the notion of weak convergence of experiments.
After an appropriate normalization the limit distributions of estimators
$T_{n}:V\rightarrow W$ under $P_{n,\vartheta}$ often have the form
$Q\ast\varepsilon_{g(\vartheta)}$ where $Q$ is some distribution
$Q$ on $\left(W,\mathcal{B}(W)\right)$. This leads to the experiment
\begin{equation}
F=\left(W,\mathcal{B}(W),\left\{ Q\ast\varepsilon_{g(\vartheta)}\;:\;\vartheta\in\Theta\right\} \right).\label{eq:2.4}
\end{equation}
Under mild regularity conditions a convolution theorem holds as we
will see below. Typically $E\geq F$ is more informative in the sense
of Le Cam. This means that $F$ is a randomization of $E$. Under
additional assumptions this can be expressed via kernels. Suppose
now that there exists a kernel $K$ 
\begin{equation}
K:\quad V\times\mathcal{B}(W)\longrightarrow[0,1]\quad,\qquad(x,B)\longmapsto K(x,B)\label{eq:2.5}
\end{equation}
with 
\begin{equation}
Q\ast\varepsilon_{g(\vartheta)}(\cdot)=K\left(P\ast\varepsilon_{h(\vartheta)}\right)(\cdot):=\int K(x,\cdot)\, P\ast\varepsilon_{h(\vartheta)}\,(dx)\label{eq:2.6}
\end{equation}
for all $\vartheta\in\Theta$. The construction of kernels is done
in Section \ref{sec:4} for standard Borel spaces.

Let $\Pi_{1}:\; V\times W\to V$ and $\Pi_{2}:\; V\times W\to W$
denote the projections. Then we find a distribution $\mu_{\vartheta}$
on $\left(V\times W,\mathcal{B}(V)\otimes\mathcal{B}(W)\right)$ with
\begin{equation}
\mathcal{L}\left(\Pi_{1}|\mu_{\vartheta}\right)=P\ast\varepsilon_{h(\vartheta)}\quad,\qquad\mathcal{L}\left(\Pi_{2}|\mu_{\vartheta}\right)=Q\ast\varepsilon_{g(\vartheta)}\label{eq:2.7}
\end{equation}
and conditional law $\mathcal{L}\left(\Pi_{2}|\Pi_{1}=x\right)=K(x,\cdot)$.
For $A\in\mathcal{B}(V)$ and $B\in\mathcal{B}(W)$ we may define
\begin{equation}
\mu_{\vartheta}(A\times B)=K\times\left(P\ast\varepsilon_{h(\vartheta)}\right)(A\times B):=\int_{A}K(x,B)\, P\ast\varepsilon_{h(\vartheta)}\, dx.\label{eq:2.8}
\end{equation}

\begin{thm}
\label{extended_convolution_theorem} Let $\Theta_{1}\subset\Theta$
be a subset with $\Theta_{1}+\Theta_{1}\subset\Theta$ such that $\left\{ P\ast\varepsilon_{h(\vartheta)}\;:\;\vartheta\in\Theta_{1}\right\} $
is boundedly complete. Then the following assertions hold. 

(a) The distribution $\nu:=\mathcal{L}\left(\Pi_{2}-f\left(\Pi_{1}\right)|\mu_{\vartheta}\right)$
does not depend on the $\vartheta\in\Theta_{1}$, i.e.~$\Pi_{2}-f\left(\Pi_{1}\right)$
is an ancillary statistic w.r.t.~$\Theta_{1}$. Moreover, $\Pi_{1}$
and $\Pi_{2}-f\left(\Pi_{1}\right)$ are $\mu_{\vartheta}$-independent
for all $\vartheta\in\Theta_{1}$. 

(b) For each $\vartheta\in\Theta$ we have 
\begin{equation}
Q\ast\varepsilon_{g(\vartheta)}=\nu\ast\mathcal{L}(f|P)\ast\varepsilon_{g(\vartheta)}.\label{eq:2.9}
\end{equation}
\end{thm}
\begin{rem}
Under the present assumptions the distribution of the randomized estimator
$x\mapsto K(x,\cdot)$ of the functional $g$ of (\ref{eq:2.2}) is
more spread out than the distribution of the point estimator $x\mapsto f(x)$.
The kernel $K$ of (\ref{eq:2.5}) may be chosen to be the convolution
kernel $K(x,\cdot)=\nu\ast\varepsilon_{f(x)}$ and it is unique $P\ast\varepsilon_{h(\vartheta)}$
a.e. for all $\vartheta\in\Theta_{1}$. \end{rem}
\begin{proof}
Without restrictions we may assume that $V_{0}=V$ and $g=f$ hold
with the identity $h$ on $V_{0}$. Choose $\vartheta,\tau\in\Theta_{1}$.
In a first step we will prove that 
\begin{equation}
K(x,\cdot)=K(x+\vartheta,\cdot)\ast\varepsilon_{-f(\vartheta)}\quad,\qquad P\ast\varepsilon_{\tau}\quad\mbox{a.e.}\label{eq:2.10}
\end{equation}
holds. To see this we consider $B\in\mathcal{B}(W)$. By (\ref{eq:2.6})
we obtain 
\begin{eqnarray}
Q\ast\varepsilon_{f(\tau)}(B) & = & Q\ast\varepsilon_{f(\tau+\vartheta)}(B+f(\vartheta))\label{eq:2.11}\\
 & = & \int K(x,B+f(\vartheta)\, P\ast\varepsilon_{\tau+\vartheta}\,(dx)\nonumber \\
 & = & \int K(x+\vartheta,B+f(\vartheta))\, P\ast\varepsilon_{\tau}\,(dx).\nonumber 
\end{eqnarray}
On the other hand $Q\ast\varepsilon_{f(\tau)}(B)=\int K(x,B)\, P\ast\varepsilon_{\tau}\,(dx)$
holds for all $\tau\in\Theta_{1}$. Bounded completness now implies
equality in (\ref{eq:2.10}) for the restriction on countable families
of $B$'s. Since $\mathcal{B}(W)$ is countably generated the result
holds for the whole $\sigma$-algebra.

Fix now some $\tau\in\Theta_{1}$. Then it is easy to see that the
projection $\Pi_{1}$ is sufficient and boundedly complete w.r.t.~to
the model $\left\{ \mu_{\vartheta+\tau}\;:\;\vartheta\in\Theta_{1}\right\} $
of (\ref{eq:2.7}) on $V\times W$.

Next we will show that $\Pi_{2}-f\left(\Pi_{1}\right)$ is an ancillary
statistic w.r.t.~$\Theta_{1}$. Consider $A\in\mathcal{B}(V)$, $B\in\mathcal{B}(W)$
and an arbitrary $\vartheta\in\Theta_{1}$. Taking (\ref{eq:2.10})
into account we have 
\begin{eqnarray}
 &  & \mu_{\tau}\left(\left\{ \Pi_{1}+\vartheta\in A,\Pi_{2}+f(\vartheta)\in B\right\} \right)\label{eq:2.12}\\
 & = & \int1_{A}(x+\vartheta)\, K(x,B-f(\vartheta))\, P\ast\varepsilon_{\tau}\,(dx)\nonumber \\
 & = & \int1_{A}(x+\vartheta)\, K(x+\vartheta,B)\, P\ast\varepsilon_{\tau}\,(dx)\nonumber \\
 & = & \mu_{\tau+\vartheta}\left(\left\{ \Pi_{1}\in A,\Pi_{2}\in B\right\} \right).\nonumber 
\end{eqnarray}
If we now write $\Pi_{2}-f(\Pi_{1})=\left(\Pi_{2}+f(\vartheta)\right)-f\left(\Pi_{1}+\vartheta\right)$
we see that 
\[
\mathcal{L}\left(\Pi_{2}-f\left(\Pi_{1}\right)|\mu_{\tau+\vartheta}\right)=\mathcal{L}\left(\Pi_{2}-f\left(\Pi_{1}\right)|\mu_{\tau}\right)
\]
does not depend on $\vartheta\in\Theta_{1}$. Basu's theorem, see
Pfanzagl (1994), p.45, then implies the independence of the sufficient
and boundedly complete statistic $\Pi_{1}$ and the ancillary statistic
$\Pi_{2}-f\left(\Pi_{1}\right)$.

Assertion (b) is then obvious since (\ref{eq:2.7}) and $\Pi_{2}=\left(\Pi_{2}-f\left(\Pi_{1}\right)\right)+f\left(\Pi_{1}\right)$
holds. Observe that (\ref{eq:2.9}) holds for all $\vartheta\in\Theta$
when we have equality for at least one $\vartheta$. 
\end{proof}
For the rest of this section we like to study applications of Theorem
\ref{extended_convolution_theorem}. First we give sufficient conditions
for bounded completeness of experiments on sample spaces with infinite
dimension.
\begin{lem}
\label{lem:2.3} Let $\mathcal{A}_{n}\subset\mathcal{A}$ denote an
increasing sequence of $\sigma$-fields with $\mathcal{A}_{0}:=\sigma\left(\mathcal{A}_{n}\;:\; n\in\mathbb{N}\right)$.
Assume that there exists an increasing sequence of $\Theta_{n}\subset\Theta$
of parameter sets such that the $\sigma$-fields $\mathcal{A}_{n}$
are sufficient and boundedly complete w.r.t. $\left(\Omega,\mathcal{A},\left\{ P_{\vartheta}\;:\;\vartheta\in\Theta_{n}\right\} \right)$
for each $n$. If $\left\{ P_{\vartheta}\;:\;\vartheta\in\Theta\right\} \ll\left\{ P_{\vartheta}\;:\;\vartheta\in\bigcup_{n=1}^{\infty}\Theta_{n}\right\} $
is dominated then $\mathcal{A}_{0}$ is boundedly complete w.r.t.
$\left\{ P_{\vartheta}\;:\;\vartheta\in\Theta\right\} $. \end{lem}
\begin{proof}
It is sufficient to prove boundedly completness for $\Theta_{0}=\bigcup_{n=1}^{\infty}\Theta_{n}$.
Consider a bounded $\mathcal{A}_{0}$-measurable function $f:\;\Omega\to\mathbb{R}$
with $E_{\vartheta}(f)=0$ for all $\vartheta\in\Theta_{0}$. By the
assumption of sufficiency there exists for each $n\in\mathbb{N}$
a version of the conditional expectation 
\[
f_{n}:=E_{\centerdot}(f|\mathcal{A}_{n})=E_{P_{\vartheta}}(f|\mathcal{A}_{n})\; P_{\vartheta}\:\mbox{a.e.}
\]
which is independent of $\vartheta\in\Theta_{n}$. Thus $\int f_{n}dP_{\vartheta}=0$
holds for all $\vartheta\in\Theta_{n}$. We conclude $f_{n}=0$ $P_{\vartheta}$
a.e. for all $\vartheta\in\Theta_{n}$, since the $\sigma$-field
$\mathcal{A}_{n}$ is boundedly complete w.r.t. $\left\lbrace P_{\vartheta}:\vartheta\in\Theta_{n}\right\rbrace $.
For fixed $\vartheta_{0}\in\Theta_{m}$ then $f_{n}$ vanishes $P_{\vartheta_{0}}$
a.e. for all $n\geq m$. On the other hand the martingale convergence
theorem implies $f_{n}\rightarrow f$ $P_{\vartheta_{0}}$ a.e. and
$f$ vanishes $P_{\vartheta_{0}}$ a.e.\end{proof}
\begin{rem}
\label{rem:3.4}

(a) Let $P\ll\lambda^{d}$ be absolutely continuous on $\mathbb{R}^{d}$
without zeros $\widehat{P}\neq0$ of the Fourier transform. If the
parameter space $\Theta\subset\mathbb{R}^{d}$ is dense in $\mathbb{R}^{d}$
then the shift family $\left\lbrace P\ast\varepsilon_{\vartheta}:\vartheta\in\Theta\right\rbrace $
is boundedly complete. Recall from Hewitt and Ross (1963), Sect. 19
and 20, that $\vartheta\mapsto\frac{dP\ast\varepsilon_{\vartheta}}{d\lambda^{d}}$
is continuous in $L_{1}\left(\lambda^{d}\right).$ Since $\Theta$
is dense our experiment is boundedly complete iff the full shift experiment
with $\vartheta\in\mathbb{R}^{d}$ is boundedly complete. Under $\widehat{P}\neq0$
the Wiener closure theorem implies the result, see Rudin (1991), Sect.
9.1--9.8. 

(b) Boll's convolution theorem can easily be extended to arbitrary
distributions $P$ on $\mathbb{R}^{d}$ with $\widehat{P}\neq0$ which
may not be absolutely continuous. The equality $Q\ast\varepsilon_{\vartheta}=K(P\ast\varepsilon_{\vartheta})$
for all $\vartheta\in\mathbb{R}^{d}$ leads to $Q\ast N\ast\varepsilon_{\vartheta}=K(P\ast N\ast\varepsilon_{\vartheta})$
for the standard normal distribution $N$ on $\mathbb{R}^{d}$. Then
$\widehat{Q}\widehat{N}=\widehat{\nu}\widehat{P}\widehat{N}$ implies
the convolution theorem.

(c) Suppose that $W$ is a topological vector space and let the shift
family $y\mapsto Q\ast\varepsilon_{y}$ be weakly continuous on the
closure $\overline{g(\Theta)}$ in $W$. Then the kernel representation
(\ref{eq:2.6}) can be extended to enlarged parameter sets $\widetilde{\Theta},\Theta\subset\widetilde{\Theta}\subset V_{0}$.
A sufficient condition is 
\begin{equation}
g(\widetilde{\Theta})\subset\overline{g(\Theta)}.\label{suff_condition}
\end{equation}
The proof follows by continuity arguments. Define a new family $Q'_{\vartheta}:=K(P\ast\varepsilon_{h(\vartheta)})$
for all $\vartheta\in\widetilde{\Theta}$ on $W$. Since $Q'_{\vartheta}=Q\ast\varepsilon_{g(\vartheta)}$
holds for a dense set of parameters $g(\vartheta),\vartheta\in\Theta,$
the distributions $\left(Q'_{\vartheta}\right)_{\vartheta\in\widetilde{\Theta}}$
must belong to a shift family. \end{rem}
\begin{example}
\label{ex:2.4} Let $\left(\Omega,\mathcal{A},\left\lbrace P_{h}:h\in H\right\rbrace \right)$
a Gaussian shift experiment with likelihood ratio 
\begin{equation}
\frac{dP_{h}}{dP_{0}}=\exp\left(L(h)-\frac{{\Vert h\Vert}^{2}}{2}\right),\; h\in H\label{g1}
\end{equation}
see Strasser (1985), chap. 11, where $(H,\langle\cdot,\cdot\rangle)$
denotes a separable real Hilbert space and $h\mapsto L(h)$ is a centered
linear Gaussian process w.r.t. $P_{0}$ and covariance $\mbox{Cov}_{P_{0}}(L(h),L(g))=\langle h,g\rangle$
for all $h,g\in H$. Let $(g_{i})_{i\in N}$, $N\subset\mathbb{N},$
denote a countable family in $H$ and let $\mathcal{A}_{0}=\sigma(L(g_{i}):i\in N)$
denote the induced $\sigma$-field on $\Omega$. Then we have: 

(a) If $\Theta\subset H$ is a subset of the sub-Hilbert space generated
by $(g_{i})_{i\in N}$, then $\mathcal{A}_{0}$ is sufficient for
the experiment $E=\left(\Omega,\mathcal{A},\left\lbrace P_{h}:h\in\Theta\right\rbrace \right)$. 

(b) Consider for each $i\in N$ a subset $\Pi_{i}\subset\mathbb{R}$
with $\stackrel{\circ}{\overline{\Pi}}_{i}\neq\emptyset$. Suppose
that the parameter space is rich enough in the sense that 
\begin{equation}
\left\lbrace \sum_{i\in J}\alpha_{i}g_{i}:\alpha_{i}\in\Pi_{i},\: J\subset N,\: J\:\mbox{finite}\right\rbrace \subset\Theta\label{g2}
\end{equation}
holds. Then $\mathcal{A}_{0}$ is boundedly complete for $E$. 

In order to give a proof of (a) observe that for each parameter $h\in H$
with Hilbert space representation $h=\sum_{i\in N}\alpha_{i}g_{i}$
we may choose densities $L(h)=\sum_{i\in N}\alpha_{i}L(g_{i})$ such
that (\ref{g1}) becomes $\mathcal{A}_{0}$-measurable. Thus it is
easy to see that all densities (\ref{g1}) are $\mathcal{A}_{0}$-measurable
for $h\in\Theta$. \\
 Part (b) follows from Lemma \ref{lem:2.3}. Without restriction we
may assume that the elements $(g_{i})_{i\in N}$ are linearly independent
in $H$. Otherwise we may cancel some members and $\mathcal{A}_{0}$
remains unchanged. For finite $J\subset N$ choose $\mathcal{A}_{J}=\sigma\left(g_{i}:i\in J\right)$
and $\Theta_{J}:=\left\lbrace \sum_{i\in J}\alpha_{i}g_{i}:\alpha_{i}\in\Pi_{i}\right\rbrace .$
Obviously, $\mathcal{A}_{J}$ is sufficient w.r.t. $\left\lbrace P_{h}:h\in\Theta_{J}\right\rbrace $
by a proper choice of the densities $L(h)$. The bounded completeness
of $\mathcal{A}_{J}$ can be proved as follows. The experiment 
\[
\left(\mathbb{R}^{J},\mathcal{B}^{J},\left\lbrace \mathcal{L}\left(\left(L(g_{i})\right)_{i\in J}|h\right):h\in\Theta_{J}\right\rbrace \right)
\]
is an exponential family of normal distributions with non-singular
covariance matrix on $\mathbb{R}^{J}$. The family is boundedly complete
since $\times_{i\in J}\overline{\Pi}_{i}$ has an inner points. 
\end{example}
\begin{example}\label{examp:3:6}

(a) Let $P_{\vartheta}:=\mathcal{L}((X_{1}(t))_{0\leq t\leq1}|\vartheta)$
denote the distribution of the signal plus noise model (\ref{signal_noise})
on $C[0,1]$ and consider the Hilbert space $L_{2}[0,1]$ as parameter
space. Via the injection $\vartheta\mapsto h(\vartheta)$ given by
(\ref{signal}) the family is a shift family 
\begin{equation}
P_{\vartheta}=P_{0}\ast\varepsilon_{h(\vartheta)},\:\vartheta\in L_{2}[0,1],\label{sn1}
\end{equation}
where $P_{0}$ denotes the Wiener measure on $C[0,1]$. The present
family (\ref{sn1}) is a Gaussian shift experiment where the densities
(\ref{g1}) are determined by the stochastic integral 
\begin{equation}
L(\vartheta)=\int_{0}^{1}\vartheta(s)\, B_{1}(ds)\label{sn2}
\end{equation}
w.r.t. Brownian motion. Observe, that the form of the densities follows
from the well-known Girsanov formula. Example \ref{ex:2.4} provides
conditions for the completeness of $\mathcal{B}(C[0,1])$ w.r.t. $E=\left\lbrace P_{\vartheta}:\vartheta\in\Theta\right\rbrace $.
We will only give a simple example. Let 
\begin{equation}
\Theta_{1}\:\mbox{denote the polynomials on}\:[0,1]\label{sn3}
\end{equation}
with non-negative rational coefficients. If 
\begin{equation}
\Theta_{1}\subset\Theta\label{sn4}
\end{equation}
holds, then $\mathcal{B}(C[0,1])$ is boundedly complete for $E$
and the assumptions of Theorem \ref{extended_convolution_theorem}
are fullfilled since $\Theta_{1}+\Theta_{1}=\Theta_{1}$ holds. Observe
that $\Theta_{1}$ generates the whole Hilbert space $L_{2}[0,1]$
and thus the projections $f\mapsto f(t)$ on $C[0,1]$ are $P_{0}$
a.e. measurable w.r.t. to $\mathcal{A}_{0}$. To see this consider
$\vartheta=1_{[0,t]}$.

Let now $T:\; C[0,1]\to W$ be an estimator of a function $g:\;\Theta\to W$,
$g=f\circ h_{|\Theta}$ defined in (\ref{eq:2.2}) which factorizes
via our signal $h$. Suppose that $T$ is an equivariant estimator
in law, i.e. 
\[
\mathcal{L}(T(X(\cdot))|\vartheta)=\mathcal{L}(T(B(\cdot)))\ast\varepsilon_{g(\vartheta)}=:Q\ast\epsilon_{g(\vartheta)}
\]
for each $\vartheta\in\Theta$. Then the estimator $T_{0}:=f(X(\cdot))$
is superior in the sense that $Q=\nu\ast\mathcal{L}\left(T_{0}|\vartheta=0\right)$
holds. 

(b) As application the convolution theorem (\ref{signal_noise_convolution})
will be established for arbitrary sample size $n$. Assume for instance
that $\vartheta_{0}+\Theta_{1}\subset\Theta$ holds where $\Theta_{1}$
denotes again the polynomials of (\ref{sn3}). Due to the translation
invariance of the problem we may assume that $\vartheta_{0}=0$ holds.
In a first step the sample size $n$ will be reduced. By (\ref{sn2})
the product densities are given by the stochastic integral w.r.t.
to the processes (\ref{signal_noise}) 
\begin{equation}
\frac{dP_{\vartheta}^{n}}{dP_{0}^{n}}=\exp\left(\int_{0}^{1}\vartheta(s)\left(\sum_{i=1}^{n}X_{i}\right)\, ds+n\frac{\Vert\vartheta\Vert^{2}}{2}\right).\label{sn5}
\end{equation}
Thus $\sqrt{n}S$ is sufficient where $S$ is the mean statistic (\ref{natural_estimator}).
If we turn to the local parameterization (\ref{another_local_parameter})
we have 
\begin{equation}
\mathcal{L}\left(\sqrt{n}S\left|P_{\eta/{\sqrt{n}}}^{n}\right.\right)=\mathcal{L}(X_{1}(\cdot)|P_{\eta})\label{sn6}
\end{equation}
on $C[0,1]$. Since $\sqrt{n}S$ is sufficient there exists a kernel
\begin{equation}
C:C[0,1]\times\mathcal{B}\left(C[0,1]^{n}\right)\rightarrow[0,1]\label{sn7}
\end{equation}
which reproduces the product measures from the image distributions
\begin{equation}
P_{\eta/\sqrt{n}}^{n}(\cdot)=\int C(x,\cdot)\, dP_{\eta}(x),\label{sn8}
\end{equation}
see Pfanzagl (1994), Prop. 1.3.1 for instance. For $g=h$ we now find
the kernel required in (\ref{eq:2.6}), namely 
\begin{equation}
Q_{h(\eta)}(\cdot):=\mathcal{L}\left(T\left|P_{\eta/\sqrt{n}}^{n}\right.\right)(\cdot)=\int C(x,T^{-1}(\cdot))\, P_{0}\ast\varepsilon_{h(\vartheta)}(dx)\label{sn9}
\end{equation}
Thus Theorem \ref{extended_convolution_theorem} implies the convolution
theorem (\ref{signal_noise_convolution}) and $S$ is the best equivariant
estimator, see Pötzelberger et al. (2000) for more details about equivariant
estimation. \end{example}

\begin{example}\label{example_levy_pro} The treatment of the Levy
processes of Example \ref{levy_pre_version} is based on a completeness
result about countable product experiments 
\begin{equation}
\left(\times_{i=1}^{\infty}\Omega_{i},\otimes_{i=1}^{\infty}\mathcal{A}_{i},\left\lbrace \otimes_{i=1}^{\infty}P_{s_{i}}^{(i)}:(s_{i})_{i}\in\Theta\right\rbrace \right)\label{levy1}
\end{equation}
where $\Theta\subset\times_{i=1}^{\infty}\Theta_{i}$ is a suitable
parameter space. Suppose that each factor $\left(\Omega_{i},\mathcal{A}_{i},\left\lbrace P_{s_{i}}^{(i)}:s_{i}\in\Theta_{i}\right\rbrace \right)$
is boundedly complete. Suppose that $s_{0}=\left(s_{i}^{0}\right)_{i\in\mathbb{N}}\in\Theta$
denotes a fixed parameter. Then Lemma \ref{lem:2.3} immediately implies
that (\ref{levy1}) is boundedly complete for 
\begin{equation}
\Theta=\left\lbrace \left(s_{i}\right)_{i\in\mathbb{N}}\in\times_{i=1}^{\infty}\Theta_{i}:s_{i}=s_{i}^{(0)}\:\mbox{finally}\:\right\rbrace ,\label{levy2}
\end{equation}
see also Landers and Rogge (1976) for related argument for the completeness
of product experiments. Let us now turn to Example \ref{levy_pre_version}.
Define next $t_{0}=0$, $\vartheta_{0}^{(0)}=0$ and $\vartheta_{0}=0$.
Then the transformation 
\begin{equation}
X\rightarrow\left(Y_{i}+s_{i}\right)_{i\in\mathbb{N}}\label{levy3}
\end{equation}
given by $Y_{i}:=Z_{t_{i}}-Z_{t_{i-1}}$, $s_{i}=\vartheta_{i}-\vartheta_{i-1}$,
denotes a one to one transformation of the shifted Levy process (\ref{levy_observations})
to the independent increment processes of the right hand side of (\ref{levy3}).
Observe that by Wiener's closure theorem the family $\left(\mathcal{L}\left(Y_{i}+s_{i}\right)\right)_{s_{i}\in\mathbb{Q}}$
is boundedly complete for each $i$ since the Fourier transforms of
absolutely continuous distributions do not vanish, see Remark \ref{rem:3.4}.
Thus the Levy process model is boundedly complete iff the product
model of the right hand side of (\ref{levy3}) is. Our results above
imply that this results is true if we take $s_{i}^{(0)}:=\vartheta_{i}^{(0)}-\vartheta_{i-1}^{(0)}$
and if we restrict ourselves to shifts $\left(\vartheta_{i}\right)_{i\in\mathbb{N}}$
of the Levy process with $\vartheta_{i}=\vartheta_{i}^{(0)}$ finally.
The convolution theorem (\ref{ct:3}) follows again from Theorem \ref{extended_convolution_theorem}.
Note that we may assume without restrictions (after applying a shift
to $Y_{i}$) that $\vartheta^{0}=0\in\mathbb{R}^{\mathbb{N}}$ holds.
Then the choice $\Theta_{1}=\left\lbrace \left(\vartheta_{i}\right)_{i\in\mathbb{N}}\in\mathbb{Q}^{\mathbb{N}}:\vartheta_{i}=0\:\mbox{finally}\:\right\rbrace $
is appropriate. \end{example}
 
\begin{rem} The present convolution
theorem can be extended to differentiable statistical functionals
$\kappa:\Theta\rightarrow\mathbb{R}$ in the sense of van der Vaart
(1988,1989,1991). For LAN families the linearization of $\kappa$
at $\vartheta_{0}$ via the canonical gradient implies the result.
\end{rem} 

\begin{example}\label{ex:3:8} This example shows that
generally the convolution factors are not unique on $\mathbb{R}$
when the assumptions of Wiener's closure theorem do not hold. In particular
we present different probability measures $\mu,\nu$ such that the
equality $\mu\ast\eta=\nu\ast\eta$ holds for each probability measure
$\eta$ with Fourier transform vanishing outside of the interval $[-1,1]$.
Let $\nu$ be a probability measure on $(\mathbb{R},\mathcal{B})$
with Lebesgue density $f:x\mapsto\frac{2\sin^{2}\left(\frac{x}{2}\right)}{\pi x^{2}}.$
It is easy to show that $\nu$ possesses the Fourier transform $\widehat{\nu}(t)=(1-|t|)\mathbf{1}_{[-1,1]}(t),$
see Gnedenko (1968), p. 236. Next we regard the probability measure
\[
\mu_{0}=\frac{1}{2}\varepsilon_{0}+\sum_{k=1}^{\infty}\frac{2}{\pi^{2}(2k+1)^{2}}\left(\varepsilon_{4k+2}+\varepsilon_{-(4k+2)}\right),
\]
which $\pi$-periodical Fourier transform is $\widehat{\mu_{0}}(t)=1-\frac{2}{\pi}|t|$
for $|t|\leq\frac{\pi}{2}$, see Renyi (1966), p. 271. The probability
measure $\mu$ is defined as distribution of $\mu_{0}$ which is scaled
with the factor $\frac{\pi}{2}$. The Fourier transform of $\mu$
is then $\widehat{\mu}(t)=1-|t|$ for $|t|\leq1$. Thus we obtain
$\mu\ast\eta=\nu\ast\eta$ for each probability measure $\eta$, which
Fourier transform have a support inside of the interval $[-1,1]$.

\end{example}

\section{Convergence to limit experiments}

\label{sec:4}

In this section it is pointed out how the convergence of statistics
and the convergence of experiments are related. We present an elementary
proof of the main Theorem \ref{main_convolution_theorem} for distributions
$Q_{\vartheta}$ on standard Borel spaces. This theorem is a special
case of more general results of Le Cam, see Le Cam and Yang (2000),
Le Cam (1994), Th. 1 and also Strasser (1985) or Torgersen (1991).
Let $E_{n}=\left(\Omega_{n},\mathcal{A}_{n},\left\lbrace P_{n,\vartheta}:\vartheta\in\Theta_{n}\right\rbrace \right)$
denote as in (\ref{eq:1.5}) a statistical experiment and let $E=\left(\Omega,\mathcal{A},\left\lbrace P_{\vartheta}:\vartheta\in\Theta\right\rbrace \right)$
be another experiment with $\Theta_{n}\uparrow\Theta$. Recall from
Strasser (1985), Sect. 60, that $E_{n}$ is said to be weakly convergent
to $E$, if all finite dimensional marginals distributions of the
loglikelihood processes 
\[
\mathcal{L}\left(\left.\left(\log\frac{dP_{n,t}}{dP_{n,s}}\right)_{t\in I}\right|P_{n,s}\right)\rightarrow\mathcal{L}\left(\left.\left(\log\frac{dP_{t}}{dP_{s}}\right)_{t\in I}\right|P_{s}\right)
\]
weakly converge for all $s\in\Theta$ and all $I\subset\Theta$, $\left|I\right|<\infty.$

Let $D$ denote a standard Borel space and $\mathcal{D}$ its Borel
$\sigma$-field on $D$.

\begin{thm}\label{main_convolution_theorem} Suppose that the
sequence of experiments $E_{n}$ converges weakly to a limit experiment
$E$. Let 
\begin{equation}
T_{n}:\Omega_{n}\rightarrow D\label{}
\end{equation}
be a sequence of statistics with values in a standard Borel space
$D$ such that 
\begin{equation}
\mathcal{L}\left(\left.T_{n}\right|P_{n,\vartheta}\right)\rightarrow Q_{\vartheta}\label{schwache_kogz_4_3}
\end{equation}
weakly converges to the some distribution $Q_{\vartheta}$ on $D$
for all $\vartheta\in\Theta$. Then there exists a kernel $K:\Omega\times\mathcal{D}\rightarrow\left[0,1\right]$
with 
\begin{equation}
Q_{\vartheta}=\int K\left(x,\cdot\right)dP_{\vartheta}\label{kernel}
\end{equation}
for all $\vartheta\in\Theta$. In particular, $E$ is more informative
than $\left(D,\mathcal{D},\left\lbrace Q_{\vartheta}:\vartheta\in\Theta\right\rbrace \right)$.
\end{thm}
 \begin{proof} Recall that on standard Borel spaces
the set of probability measures is a separable metric space w.r.t.
the topology of weak convergence. Thus we may choose a countable dense
subset $\left\lbrace Q_{\vartheta_{j}}:j\in\mathbb{N}\right\rbrace $
of $\left\lbrace Q_{\vartheta}:\vartheta\in\Theta\right\rbrace $.
We will identify by $Q_{\vartheta_{j}}=:Q_{j}$ and $P_{\vartheta_{j}}=:P_{j}$.
Introduce the additional distributions 
\begin{equation}
Q_{0}=\sum_{j=1}^{\infty}\frac{Q_{j}}{2^{j}}\;\mbox{and}\; P_{0}=\sum_{j=1}^{\infty}\frac{P_{j}}{2^{j}}.\label{dominating_measure}
\end{equation}
Write shortly $P_{n,j}:=P_{n,\vartheta_{j}}$. Define 
\begin{equation}
P_{n,0}:=a_{n}\sum_{j=1}^{\infty}\frac{P_{n,j}}{2^{j}}\mathbf{1}_{\Theta_{n}}(\vartheta_{j}),\label{eq:4:5}
\end{equation}
where $a_{n}$ are normalizing constants. The proof of Theorem \ref{main_convolution_theorem}
relies on two lemmas. First we add $P_{n,0}$ and $P_{0}$ to our
experiments.

\begin{lem}\label{lem:4.2} If $E_{n}\rightarrow E$ converges
weakly then also 
\[
\left(\Omega_{n},\mathcal{A}_{n},\left\lbrace P_{n,j}:j\in\mathbb{N}_{0}\right\rbrace \right)\rightarrow\left(\Omega,\mathcal{A},\left\lbrace P_{j}:j\in\mathbb{N}_{0}\right\rbrace \right)
\]
converges weakly. \end{lem}
 \begin{proof} Consider the subexperiments
for $I=\left\lbrace 1,\ldots,k\right\rbrace $ and add $P'_{n,0}:=\sum\limits _{j=1}^{m}c_{j}P_{n,j}$
and $P'_{0}:=\sum\limits _{j=1}^{m}c_{j}P_{j}$, where $0<c_{j}<1$,
$\sum\limits _{j=1}^{m}c_{j}=1$, $m\in\mathbb{N}$ for $n$ large
enough. It is easy to see that the experiments $\left\lbrace P'_{n,0},P_{n,1},\ldots,P_{n,k}\right\rbrace $
weakly converge to $\left\lbrace P'_{0},P_{1},\ldots,P_{k}\right\rbrace $,
since their likelihood processes can be expressed by linear dependence
by the likelihood processes of $E_{n}$ and $E$. If $m$ tends to
infinity we find coefficients $c_{j}$ such that $P'_{n,0}$ tends
to $P_{n,0}$ uniformly in $n$ w.r.t. the norm of total variation.
Thus $P'_{n,0}$ and $P'_{0}$ may be substituted by $P_{n,0}$ and
$P_{0}$ and the convergence of experiments carries over. \end{proof}
In the case of Theorem \ref{main_convolution_theorem} we have \begin{lem}\label{lem:4:4} 

(a) There exist a subsequence $\left\lbrace m\right\rbrace \subset\mathbb{N}$
and a probability measure $\mu_{0}$ on $\mathbb{R}^{\mathbb{N}}\times D$
such that 
\begin{equation}
\mathcal{L}\left(\left.\left(\left(\log\frac{dP_{m,j}}{dP_{m,0}}\right)_{j\in\mathbb{N}},T_{m}\right)\right|P_{m,0}\right)\rightarrow\mu_{0}\label{eq:4:6}
\end{equation}
weakly converges as $m\rightarrow\infty$.

(b) Under $P_{m,j}$ the limit law $\mu_{j}$ of (\ref{eq:4:6}) exists.
We have $\mu_{j}\ll\mu_{0}$ with density 
\begin{equation}
\frac{d\mu_{j}}{d\mu_{0}}\left(x,d\right)=\exp\left(p_{j}\left(x\right)\right)\label{density}
\end{equation}
on $\mathbb{R}^{\mathbb{N}}\times D$ where $p_{j}\left(x\right)$
denote the $j$-th coordinate of $x\in\mathbb{R}^{\mathbb{N}}$.  
\end{lem} 

\begin{proof} Part (a) follows from the tightness of
the marginals. Part (b) is a consequence of the third Lemma of Le
Cam, see van der Vaart (1988), Appendix A1 , Janssen (1998), Sect.
14 or Jacod and Shiryaev (2003). Observe, that Le Cam's third Lemma
also holds for non-Gaussian limit experiments. Only contiguity of
$\left(P_{m,j}\right)_{m}$ w.r.t $\left(P_{m,0}\right)_{m}$ is required.
This condition follows from Lemma \ref{lem:4.2} and the first Lemma
of Le Cam. \end{proof}

The proof of Theorem \ref{main_convolution_theorem} can be completed
by the following arguments. Consider the canonical projections $\pi_{1}:\mathbb{R}^{\mathbb{N}}\times D\rightarrow\mathbb{R}^{\mathbb{N}}$
and $\pi_{2}:\mathbb{R}^{\mathbb{N}}\times D\rightarrow D$. Obviously,
$\pi_{1}$ is sufficient for $\left\lbrace \mu_{j}:j\in\mathbb{N}_{0}\right\rbrace $.
Thus there exists a version of the conditional distribution of $\left(\pi_{1},\pi_{2}\right)$
given $\pi_{1}$ 
\begin{equation}
C:\mathbb{R}^{\mathbb{N}}\times\left(\mathcal{B}^{\mathbb{N}}\otimes\mathcal{D}\right)\rightarrow\left[0,1\right]\label{conditional_distribution}
\end{equation}
which is independent of $j$, i.e. $\mu_{j}=\int C\left(x,\cdot\right)d\mu_{j}^{\pi_{1}}\left(x\right)$.
Now we may choose 
\begin{equation}
K\left(\omega,A\right)=C\left(\left(\log\frac{dP_{j}}{dP_{0}}\left(\omega\right)\right)_{j\in\mathbb{N}},\pi_{2}^{-1}\left(A\right)\right)\label{eq:4:9}
\end{equation}
for $A\in\mathcal{D}$ and (\ref{kernel}) holds for all $\vartheta\in\left\lbrace \vartheta_{j}:j\in\mathbb{N}\right\rbrace $.
That equality can be extended to all $\vartheta\in\Theta$ by the
following continuity arguments. Define new distributions $Q'_{\vartheta}=K(P_{\vartheta})$
for all $\vartheta\in\Theta$ via (\ref{eq:2.6}). Since $Q_{\vartheta_{j}}=Q'_{\vartheta_{j}}$
holds and $\left\lbrace Q_{\vartheta_{j}}:j\in\Theta\right\rbrace $
is dense we have $Q_{\vartheta}=Q'_{\vartheta}$ for all $\vartheta\in\Theta$
and (\ref{kernel}) holds. These arguments finish the proof of Theorem
\ref{main_convolution_theorem}. \end{proof}Observe that $\left\lbrace P_{\vartheta_{j}}:j\in\mathbb{N}\right\rbrace $
and $\left\lbrace Q_{\vartheta_{j}}:j\in\mathbb{N}\right\rbrace $
are dense in $E$ and $\left\lbrace Q_{\vartheta}:\vartheta\in\Theta\right\rbrace $,
respectively, the norm of total variation. By continuity arguments
equation (\ref{kernel}) carries over for all $\vartheta\in\Theta$.
We like to present another example where the convolution theorem holds
for a non Gaussian limit experiment.

\begin{example}{[}Estimating the endpoints of a distribution{]} Let
$f$ be a continuous smooth density on some real interval $\left[a,b\right]$
with $f(a)>0$ and $f(b)>0$ where $f$ vanishes outside of the interval.
We want to estimate the endpoints $a,b$ based on i.i.d. replications
of sample size $n$. A local model is now given by the parameterization
of the endpoints 
\begin{equation}
a+\frac{\vartheta_{1}}{n}\quad\mbox{and}\quad b+\frac{\vartheta_{2}}{n}\label{endpoints}
\end{equation}
for a suitable pair of local parameters $\vartheta=\left(\vartheta_{1},\vartheta_{2}\right)\in\Theta_{n}\subset\mathbb{R}^{2}$.
We will restrict ourselves to the following shift and scale submodel
(\ref{observations}). Let $Y_{1},Y_{2},\ldots,Y_{n}$ be i.i.d. random
variables with common density $f$. Our observations for $1\leq i\leq n$
are 
\begin{equation}
Z_{n,i}:=\left(1+\frac{\vartheta_{2}-\vartheta_{1}}{n\left(b-a\right)}\right)\left(Y_{i}-a\right)+a+\frac{\vartheta_{1}}{n}\label{observations}
\end{equation}
which have just the endpoints (\ref{endpoints}). For suitable parameters
$\vartheta=\left(\vartheta_{1},\vartheta_{2}\right)$ introduce 
\begin{equation}
P_{n,\vartheta}:=\mathcal{L}\left(\left(\left.Z_{n,1},\ldots,Z_{n,n}\right)\right|\vartheta\right)\label{E3}
\end{equation}
and let $Z_{1:n}\leq Z_{2:n}\leq\ldots\leq Z_{n:n}$ denote the order
statistics of (\ref{observations}). We will only sketch the general
results and indicate how to prove convergence of the experiments.
A rigorous proof is only given for the uniform distribution. Let $X_{1}$
and $X_{2}$ denote two independent standard exponential random variables
with $E(X_{i})$=1 for $i=1,2$. Well-known results from extreme value
theory prove that 
\begin{equation}
n\left(Z_{1:n}-a\right)\rightarrow\frac{X_{1}+\vartheta_{1}}{f(a)}\label{E4}
\end{equation}
and 
\begin{equation}
n\left(Z_{n:n}-b\right)\rightarrow\frac{-X_{2}+\vartheta_{2}}{f(b)}\label{E5}
\end{equation}
are convergent in distribution under $\vartheta$. The convergence
of these distributions can also be shown w.r.t. the norm $\Vert\cdot\Vert$
of total variation. Under regularity assumptions concerning the smoothness
of the density $f$ the extreme order statistics $\left(Z_{1:n},Z_{n:n}\right)$
are asymptotically sufficient and weak convergence of the experiments
\begin{eqnarray}
 &  & \left(\mathbb{R}^{n},\mathcal{B}^{n},\left\lbrace P_{n,\vartheta}:\vartheta=(\vartheta_{1},\vartheta_{2})\in\Theta_{n}\right\rbrace \right)\label{E6}\\
 & \rightarrow & \left(\mathbb{R}^{2},\mathcal{B}^{2},\left\lbrace \mathcal{L}\left(\frac{X_{1}+\vartheta_{1}}{f(a)},\frac{-X_{2}+\vartheta_{2}}{f(b)}\right):\vartheta=(\vartheta_{1},\vartheta_{2})\in\Theta\right\rbrace \right)\nonumber 
\end{eqnarray}
holds provided $\Theta_{n}\uparrow\Theta$. The details are figured
out for the uniform distribution only. \end{example} \begin{lem}
The convergence of experiments (\ref{E6}) holds for the uniform density
$f=\mathbf{1}_{[0,1]}$. \end{lem} \begin{proof} In this case
the extreme order statistics $\left(Z_{1:n},Z_{n:n}\right)$ are finite
sample sufficient and the experiments 
\begin{equation}
\left\lbrace P_{n,\vartheta}\right\rbrace \:\mbox{and}\:\left\lbrace \left.\mathcal{L}\left(nZ_{1:n},n\left(Z_{n:n}-1\right)\right|\vartheta\right)\right\rbrace \label{E7}
\end{equation}
are equivalent in Le Cam's sense. Thus it is sufficient to prove the
weak convergence of the latter experiment. Consider first $\vartheta=0$.
It is well-known, see Reiss (1989), Section 5.1 and p. 121, that 
\begin{equation}
\left\Vert \left.\mathcal{L}\left(nZ_{1:n},n\left(Z_{n:n}-1\right)\right\vert 0\right)-\mathcal{L}\left(X_{1},-X_{2}\right)\right\Vert \rightarrow0\label{E8}
\end{equation}
holds w.r.t. the norm of total variation. This assertion is due to
the fact that lower and upper extreme become asymptotically independent.
The same result holds under $\vartheta$ where $\mathcal{L}(X_{1},-X_{2})$
is replaced by $\mathcal{L}(X_{1}+\vartheta_{1},-X_{2}+\vartheta_{2})$,
see also (\ref{E4}) and (\ref{E5}). Note that (\ref{E8}) is equivalent
to the $L_{1}$-convergence of the densities which are easy to handle
under the present shift and scale family. It is well-known that the
convergence of distributions w.r.t. total variation implies the convergence
of the underlying experiments. \end{proof} Let now $\vartheta\in\Theta$
belong to a dense set of $\mathbb{R}^{2}$. Then the assumptions of
the convolution theorem hold. Observe that the limit experiment of
(\ref{E6}) is boundedly complete by Wiener's closure theorem, see
Remark \ref{rem:3.4}. Let $T_{n}:\mathbb{R}^{n}\rightarrow\mathbb{R}^{2}$
be a sequence of $\Theta$ regular estimators of the endpoints for
the uniform distribution at $(a,b)=(0,1)$ with 
\begin{equation}
\mathcal{L}\left(\left.n\left(T_{n}-\left(\frac{\vartheta_{1}}{n},1+\frac{\vartheta_{2}}{n}\right)\right)\right|P_{n,\vartheta}\right)\rightarrow\mathcal{L}(Q)\label{E9}
\end{equation}
in distribution for all $\vartheta\in\Theta$. Then $\mathcal{L}(Q)$
is given by a convolution product 
\begin{equation}
\mathcal{L}(Q)=\nu\ast\mathcal{L}(X_{1},-X_{2}).\label{E10}
\end{equation}
Under regularity assumption the same holds for densities $f$ where
the convolution bound arises from the limit experiment (\ref{E6}). 
\begin{rem}
\label{last_remark} 

(a) Limit experiments for densities with jumps were treated by Ibragimov
and Has'minskii (1981), Chap. V, Pflug (1983), Strasser (1985b), Sect.
19 and in the appendix of Janssen and Mason (1990), p. 215. In these
references the reader will find regularity conditions concerning the
density $f$ which ensure weak convergence of the experiments in (\ref{E6}). 

(b) If the density has only a single jump at the lower endpoint then
the limit experiment (\ref{E6}) has to be modified and it is just
$\left\lbrace \mathcal{L}\left(\frac{X_{1}+\vartheta_{1}}{f(a)}:\vartheta_{1}\right)\right\rbrace $.
For this kind of limit experiment Millar (1983), p. 157 already obtained
a convolution theorem which is similar to (\ref{E10}). 
\end{rem}
\bigskip{}
\textbf{Acknowledgments}\\
The authors like to thank W. Hazod who pointed out the counterexample
\ref{ex:3:8}.

\bigskip{}
\parbox[t]{7cm}{%
Arnold Janssen\\
 Mathematical Institute\\
 University of Düsseldorf\\
 Universitätsstrasse 1\\
 40225 Düsseldorf, Germany\\
 janssena@uni-duesseldorf.de %
} \qquad{}%
\parbox[t]{7cm}{%
Vladimir Ostrovski\\
 independent\\
 \\
 Niederbeckstrasse 23\\
 40472 Düsseldorf, Germany\\
 vladimir\_ostrovski@web.de %
} 
\end{document}